\RequirePackage{fix-cm}
\documentclass[a4paper,12pt]{article}      
\usepackage{graphicx}
\usepackage{amsfonts}
\usepackage{amsmath}
\usepackage{epstopdf}
\usepackage{cite}
\usepackage{amsthm}
\newtheorem{theorem}{Theorem}
\newtheorem{definition}{Definition}
\newtheorem{lemma}{Lemma}
\newtheorem{remark}{Remark}
\newtheorem{corollary}{Corollary}
\begin{document}

\title{A Nice Representation for a Link between Baskakov- and Sz\'{a}sz-Mirakjan-Durrmeyer Operators and their Kantorovich Variants}  

\author{Margareta Heilmann and  Ioan Ra\c{s}a }        

\maketitle
\begin{center}
Dedicated to Professor Heiner Gonska on the occasion of his 70th birthday
\end{center}
\begin{abstract}
In this paper we consider a link between Baskakov-Durrmeyer type  operators and corresponding Kantorovich type modifications of their classical variants.
We prove a useful representation for Kantorovich variants of arbitrary order which leads to a simple proof of convexity properties for the linking operators. This also solves an open problem mentioned in 
\cite{BaumannHeilmannRasa2016}. Another open problem is presented at the end of the paper.

\end{abstract}

Dedicated to Professor Heiner Gonska on the occasion of his 70th birthday
\section{Introduction}
\label{sec1}
In \cite{Paltanea2007} P\u {a}lt\u {a}nea defined a non-trivial link between genuine Bernstein-Durrmeyer operators and classical Bernstein operators, depending on a positive real parameter $\rho$. For further results we refer to \cite{Paltanea2008}, \cite{GonskaPaltanea2010}, \cite{GonskaPaltanea2010-1}, \cite{GonskaRasaStanila2014}, \cite{GonskaRasaStanila2014-1} and \cite{HeilmannRasa2017}.

In \cite{HeilmannRasa2017} the authors of this paper considered the $k$th order Kantorovich modifications and proved for natural numbers $\rho$ a representation in terms of Bernstein basis functions (see \cite[Theorem 2]{HeilmannRasa2017}).

In this paper we investigate linking operators acting on the unbounded interval $[0,\infty)$.
Linking operators for the  Sz\'{a}sz-Mirakjan case were defined by P\u {a}lt\u {a}nea in \cite{Paltanea2008} and for Baskakov type operators by Heilmann and Ra\c{s}a in \cite{HeilmannRasa2015}.

In what follows for $c \in \mathbb{R} $ we use the notations
$$
	a^{c,\overline{j}} := \prod_{l=0}^{j-1}  (a+cl) , \; a^{c,\underline{j}} := \prod_{l=0}^{j-1}  (a-cl) , \; j \in \mathbb{N}; \quad
	 a^{c,\overline{0}}= a^{c,\underline{0}} :=1
$$which can be considered as a generalization of rising and falling factorials.  This notation enables us to state the results for the different operators in a unified form.
 
In the following definitions of the operators we omit the parameter $c$ in the notations in order to reduce the necessary sub and superscripts.

Let $c \in \mathbb{R}$, $c \geq 0$, $n \in \mathbb{R}$, $n > c$ . 
Furthermore let  $\rho \in \mathbb{R}^+$, $j \in \mathbb{N}_0$, $x \in [0,\infty)$. Then the basis functions are given by
$$
	p_{n,j}(x) = \left \{ \begin{array}{cl}
	\displaystyle \frac{n^j}{j!} x^j  e^{-nx} &, \, c = 0 ,\\
	\displaystyle \frac{n^{c,\overline{j}}}{j!} x^j (1+cx)^{-\left ( \frac{n}{c}+j\right)} &, \, c > 0 .
 	\end{array} \right .
$$
By $W_n^\rho$ we denote the space of functions $ f \in L_{1,loc}[0,\infty) $ satisfying certain growth conditions, i.~e., there exist constants $M>0$, $ 0 \leq q < n\rho +c$, such that a.~e. on $ [0,\infty)$
\begin{eqnarray*}
	 |f(t)| & \leq & Me^{qt} \mbox{ for } c=0,
	\\
	 |f(t)| & \leq & Mt^{\frac{q}{c}} \mbox{ for } c>0.
\end{eqnarray*}

In the following definition we assume that $f:I_c \longrightarrow \mathbb{R}$ is given in such a way that the corresponding integrals and series are convergent.
\begin{definition}
The operators of Baskakov-type are defined by
$$
	(B_{n,\infty} f)(x) = \sum_{j=0}^{\infty} p_{n,j}(x) f \left ( \frac{j}{n} \right ) ,
$$
and the genuine Baskakov-Durrmeyer type operators are denoted  by 
\begin{eqnarray}
\label{eq0.2}
	 (B_{n,1} f)(x)
	& = & 
	p_{n,0}(x) f(0)  +
	\sum_{j=1}^{\infty} p_{n,j}(x) (n+c) \int_0^{\infty} p_{n+2c,j-1} (t) f (t) dt .
\end{eqnarray}
Depending on a parameter  $\rho \in \mathbb{R}^+$ the linking operators are given by
\begin{equation}
\label{eq0.3}
	 (B_{n,\rho} f)(x) = 
	p_{n,0} (x) f(0) 
	+ \sum_{j=1}^{\infty} p_{n,j}(x) \int_0^\infty \mu_{n,j,\rho}(t) f(t) dt ,
\end{equation}
where
$$
	\mu_{n,j,\rho} (t) = \left \{ \begin{array}{cl}
	\displaystyle \frac{(n\rho)^{j\rho}}{\Gamma (j \rho )} t^{j\rho-1}  e^{-n\rho t} &, \, c = 0  ,\\
	\displaystyle \frac{c^{j\rho}}{B \left (j\rho,\frac{n}{c}\rho+1 \right )} t^{j\rho -1} (1+ct)^{-\left ( \frac{n}{c}+j\right)\rho -1} &, \, c > 0 ,
	\end{array} \right .
$$
with Euler's Beta function
$$
	B(x,y) = \int_0^1 t^{x-1} (1-t)^{y-1} dt = \int_0^{\infty} \frac{t^{x-1}}{(1+t)^{x+y}} dt 
	= \frac{\Gamma (x) \Gamma (y)}{\Gamma (x+y)} , \, x,y >0.
$$
\end{definition}
For $ \rho \in \mathbb{R}^+$ the operators $B_{n,\rho}$ are well defined for functions $f \in W_n^\rho$ having a finite limit $f(0) = \lim_{x \to 0^+} f(x)$.

Setting  $c=0$ in (\ref{eq0.2})  leads to the  Phillips operators \cite{Phillips1954}, $c>0$ was investigated in \cite{Wagner2013}. To the best of our knowledge  $c=0$ in (\ref{eq0.3}) was first considered in \cite{Paltanea2008}.

In \cite{HeilmannRasa2015} (see also \cite{Paltanea2014, BaumannHeilmannRasa2016}) the authors considered
 the $k$-th order Kantorovich modifications of the operators $B_{n,\rho}$, namely,
$$
	B_{n,\rho}^{(k)}:=D^k \circ B_{n,\rho}\circ I_k, \, k \in \mathbb{N}_0 ,
$$
where
$D^k$ denotes the $k$-th order ordinary differential operator and
$$
	I_0 f = f, \ (I_k f)(x) = \int_0^x \frac{(x-t)^{k-1}}{(k-1)!}  f(t) dt,
	\mbox{ if } k \in \mathbb{N}.
$$
For $k=0$ we omit the superscript $(k)$ as indicated by the definition above.
These operators play an important role in the investigation of simultaneous approximation.

This general definition contains  many known operators as special cases.
For  $c=0$ we get the $k$-th order Kantorovich modification of linking operators considered in  \cite{Paltanea2014}.
For $\rho=1, k \in\mathbb{N} $ we get the Baskakov-Durrmeyer type operators 
$B_{n,1}^{(1)}$ (see \cite{MazharTotik1985} for $c=0$ and \cite[(1.3)]{Heilmann1989} for $c \geq 0$, named $M_{n+c}$ there) and the auxiliary operators $B_{n,1}^{( k)}$ considered in \cite[(3.5)]{Heilmann1992}, (named $M_{n+c,k-1}$ there).

Concerning the limit of the operators $B_{n,\rho}^{(k)}$ for $\rho \to \infty$ some results are known which are cited here.
\begin{theorem} (\cite[Theorem 4]{Paltanea2014})
\label{thm-c1}
Let $c=0$. Assume that $f: [0,\infty) \longrightarrow \mathbb{R}$ is integrable and there exist constants $M > 0 $, $q \geq 0$ such that $|f(t) | \leq M e^{qt}$ for $t \in [0,\infty)$. Then for any $b>0$ there is $\rho_0 >0$, such that $B_{n,\rho} f $ exists for all $\rho \geq \rho_0$ and we have
$$ \lim_{\rho \to \infty} (B_{n,\rho} f)(x)  = (B_{n,\infty} f)(x), \mbox{ uniformly for } x \in [0,b].$$
\end{theorem}
In \cite{HeilmannRasa2015} explicit representations for the images of polynomials for all operators $B_{n,\rho}^{(k)}$ were calculated which led to the following result for $c \geq 0$.
\begin{theorem} (\cite[Theorem 1, Theorem 2, Corollary 1]{HeilmannRasa2015})
\label{thm-c2}
For each polynomial $p$ we have
$$ \lim_{\rho \to \infty} (B_{n,\rho}^{(k)} p)(x)  = (B_{n,\infty}^{(k)} p)(x),$$
uniformly on every compact subinterval of   $[0,\infty)$.
\end{theorem}
A different function space was considered in \cite{BaumannHeilmannRasa2016} for the case $c \geq 0$.
\begin{theorem} (\cite[Lemma 5, Corollary 3]{BaumannHeilmannRasa2016})
\label{thm-c3}
Let $ f \in C^2[0,\infty)$ with $\|f''\|_{\infty}^{[0,\infty )} < \infty$. Then we have
$$ \lim_{\rho \to \infty} (B_{n,\rho} f)(x)  = (B_{n,\infty} f)(x), $$
uniformly on every compact subinterval of $[0,\infty)$.
\end{theorem}

For $\rho = 1 $ and $\rho = \infty$ nice explicit representations are known, i.e.,
\begin{eqnarray*}
	(B_{n,1}^{(k)} f)(x) 
	&=& 
	\frac{n^{c,\overline{k}}}{n^{c,\underline{k-1}}} \sum_{j=0}^{\infty} p_{n+ck,j} (x) \int_{I_c}
	p_{n-c(k-2),j+k-1} (t) f(t) dt ,
\\
	B_{n,\infty}^{(k)} (f;x) 
	&=&
	n^{c,\overline{k}} \sum_{j=0}^{\infty} p_{n+ck,j} (x) \Delta^k_{\frac{1}{n}} 
	I_k \left (f;  \frac{j}{n} \right ),
\end{eqnarray*}
where the forward difference of order $k$ with the step $h$ for a function $g$ is given by 
$\Delta^k_{h} g (x) = \sum_{i=0}^k {k \choose i} (-1)^{k-i} g(x+ih)  $. 
By using Peano's representation theorem for divided differences (see, e.g., \cite[p. 137]{Schumaker2007}) this can also be written as
\begin{equation}
\label{eq0}
	B_{n,\infty}^{(k)} (f;x) 
	=
	\frac{n^{c,\overline{k}}}{n^{k-1}} \sum_{j=0}^{\infty} p_{n+kc,j} (x) 
	\int_0^{\infty} N_{n,k,j}(t) f(t) dt,
\end{equation}
where $N_{n,k,j}$ denotes the B-spline of order $k$ to the equidistant knots
$ \frac{j}{n}, \dots \frac{j+k}{n}$, defined by
\begin{eqnarray*}
	N_{n,1,j}(t) 
	&=& 
	\left \{  \begin{array}{ccl}
	1 & , & \frac{j}{n} \leq t < \frac{j+1}{n}, \\
	0 & , & \mbox{otherwise,} \end{array}\right .
\\
	N_{n,k,j} (t)
	&=& 
	\frac{n}{k-1} \left \{ \left ( t - \frac{j}{n} \right ) N_{n,k-1,j} (t)  + 
	 \left ( \frac{j+k}{n} -t\right )N_{n,k-1,j+1} (t) \right \} .
\end{eqnarray*}
Our goal was to find useful representations also for $\rho \not= 1, \infty$ for the general case $k \in \mathbb{N}$.

Throughout this paper we will use the following well-known formulas for the basis functions.
\begin{eqnarray}
\label{eq1.1}
	p_{n,j}' (x) 
	& = &
	n \left [ p_{n+c,j-1} (x) - p_{n+c,j} (x) \right ],
\\
\label{eq1.2}
	x p_{n+c,j-1}' (x) 
	& = &
	(j-1) p_{n+c,j-1} (x) -j p_{n+c,j} (x) .
\end{eqnarray}
\section{The representation theorem}
\label{sec2}
First we treat the case $k=1$. 
In other words we prove an explicit representation for a non-trivial link between Baskakov-Durrmeyer type and Baskakov-Kantorovich operators.
Let $ f \in W_n^\rho$.
\\
We start with the definition of $B_{n,\rho}^{(1)}= D \circ B_{n,\rho} \circ I$
and consider
\begin{equation}
\label{eq1.4}
	\omega_{n,j,\rho} (t) = \left \{
	\begin{array}{lll}
	\int_t^\infty  \mu_{n,1,\rho} (u)  du &,& j=0,
	\\
	\int_0^t \left ( \mu_{n,j,\rho} (u)  -  \mu_{n,j+1,\rho} (u)  \right ) du &,& 1 \leq j .
	\end{array} \right .
\end{equation}
Thus 
$$
	\omega_{n,j,\rho}' (t) = \left \{
	\begin{array}{lll}
	- \mu_{n,1,\rho} (t)  &,& j=0,
	\\
	 \mu_{n,j,\rho} (t)  -  \mu_{n,j+1,\rho} (t)  &,& 1 \leq j .
	\end{array} \right .
$$
Note that 
\begin{equation}
\label{eq1.3}
	\omega_{n,j,\rho} (0) = 0 , \, j \in \mathbb{N} \mbox{ and }
	\lim_{t \to \infty} \omega_{n,j,\rho} (t) = 0 , \,  j \in \mathbb{N}_0 .
\end{equation}
By applying (\ref{eq1.1}), an appropriate index transform and the definition of 
$\omega_{n,j,\rho} $
we derive
\begin{eqnarray*}
	B_{n,\rho}^{(1)} (f;x)
	& = &
	\sum_{j=1}^{\infty} p_{n+c,j}' (x) \int_0^\infty \mu_{n,j,\rho} (t) I_1(f;t) dt
\\
	& = & 
	n \sum_{j=1}^{\infty} p_{n+c,j-1}(x) \int_0^\infty \mu_{n,j,\rho} (t) I_1(f;t) dt
\\
	& &
	- n \sum_{j=1}^{\infty} p_{n+c,j}(x) \int_0^\infty \mu_{n,j,\rho} (t) I_1(f;t) dt
\\
	& =  &
	n p_{n+c,0} (x) \int_0^\infty  \mu_{n,1,\rho} (t) I_1(f;t) dt 
\\
	& &
	+n \sum_{j=1}^{\infty} p_{n+c,j}(x) 
	\int_0^\infty \left [ \mu_{n,j+1,\rho} (t) -  \mu_{n,j,\rho} (t) \right ] I_1(f;t) dt 
\\
	& =  &
	-n \sum_{j=0}^{\infty} p_{n+c,j}(x) 
	\int_0^\infty  \omega_{n,j,\rho}' (t)  I_1(f;t) dt .
\end{eqnarray*}
Integration by parts and observing (\ref{eq1.3}) leads to
\begin{eqnarray}
\label{eq-r1}
	B_{n,\rho}^{(1)} (f;x)
	& = &
	n \sum_{j=0}^{\infty} p_{n+c,j}(x) 
	\int_0^\infty  \omega_{n,j,\rho} (t) f (t) dt .
\end{eqnarray}
\begin{lemma}
\label{lem-int}
Let $\rho \in \mathbb{N}$. Then
$$
	\omega_{n,j,\rho} (t) = \sum_{i=0}^{\rho-1} p_{n\rho+c,i+j\rho} (t).
$$
\end{lemma}
\begin{proof}
First we treat the case $c=0$.
\\
For $ l \in \mathbb{N}$, integration by parts leads to
\begin{eqnarray*}
	\int u^{l\rho-1} e^{-n\rho u} du
	& = &
	-(l\rho-1)! \sum_{i=0}^{l\rho-1} \frac{1}{(n\rho)^{i+1}} \cdot \frac{1}{(l\rho-1-i)!}
	u^{l\rho-1-i}e^{-n\rho u}
\\
	& = &
	-(l\rho-1)! \sum_{i=0}^{l\rho-1} \frac{1}{(n\rho)^{l\rho -i}} \cdot \frac{1}{i!}
	u^{i}e^{-n\rho u}.
\end{eqnarray*}
Thus for $j=0$
\begin{eqnarray*}
	\omega_{n,0,\rho} (t) 
	& = &
	\frac{(n\rho)^{\rho}}{(\rho-1)!} \int_t^\infty  u^{\rho-1} e^{-n\rho u} du
\\
	& = &
	 -\frac{(n\rho)^{\rho}}{(\rho-1)!} (\rho-1)!  \sum_{i=0}^{\rho-1} \left .
	\frac{1}{(n\rho)^{\rho -i}} 
	\cdot \frac{1}{i!}	u^{i}e^{-n\rho u}\right |_t^\infty
\\
	& = &
	\sum_{i=0}^{\rho-1}
	(n\rho)^{i}
	\cdot \frac{1}{i!}	t^{i}e^{-n\rho t}
\\
	& = &
	 \sum_{i=0}^{\rho-1} p_{n\rho,i} (t).
\end{eqnarray*}

For $ j \in \mathbb{N} $ we get
\begin{eqnarray*}
	\lefteqn{\omega_{n,j,\rho} (t)}
\\ 
	& = &
	\frac{(n\rho)^{j\rho}}{(j\rho-1)!} \int_0^t  u^{j\rho-1} e^{-n\rho u} du
	- \frac{(n\rho)^{(j+1)\rho}}{((j+1)\rho-1)!} \int_0^t  u^{(j+1)\rho-1} e^{-n\rho u} du
\\
	& = &
	- \frac{(n\rho)^{j\rho}}{(j\rho-1)!} \cdot (j\rho-1)! 
	\sum_{i=0}^{j\rho-1} \left .
	\frac{1}{(n\rho)^{j\rho -i}} 
	\cdot \frac{1}{i!}	u^{i}e^{-n\rho u}\right |_0^t
\\
	& &
	+ \frac{(n\rho)^{(j+1)\rho}}{((j+1)\rho-1)!} \cdot ((j+1)\rho-1)! 
	\sum_{i=0}^{(j+1)\rho-1} \left .
	\frac{1}{(n\rho)^{(j+1)\rho -i}} 
	\cdot \frac{1}{i!}	u^{i}e^{-n\rho u}\right |_0^t
\\
	& = &
	\sum_{i=j\rho}^{(j+1)\rho-1} (n\rho )^i  \frac{1}{i!}	t^{i}e^{-n\rho t}
\\
	& = &
	\sum_{i=0}^{\rho -1} p_{n\rho,i+j\rho} (t).
\end{eqnarray*}
Now we consider $ c >0$.
\\
For $ l \in \mathbb{N}$, integration by parts leads to
\begin{eqnarray*}
	\lefteqn{\int u^{l\rho-1} (1+cu)^{-(\frac{n}{c}+l)\rho -1} du}
\\
	& = &
	-\frac{(l\rho-1)!}{\Gamma((\frac{n}{c}+l)\rho +1)} \sum_{i=0}^{l\rho-1}  
	\frac{\Gamma((\frac{n}{c}+l)\rho -i)}{(l\rho-1-i)! c^{i+1}}
	u^{l\rho-1-i}  (1+cu)^{-(\frac{n}{c}+l)\rho +i} 
\\
	& = &
	-\frac{(l\rho-1)!}{\Gamma((\frac{n}{c}+l)\rho +1)}  \sum_{i=0}^{l\rho-1}  
	\frac{\Gamma(\frac{n}{c}\rho +1+i)}{i!c^{l\rho -i}} 
	u^{i}  (1+cu)^{-(\frac{n}{c}\rho +1 +i)}. 
\end{eqnarray*}
Thus for $j=0$
\begin{eqnarray*}
	\lefteqn{\omega_{n,0,\rho} (t)}
\\ 
	& = &
	\frac{c^\rho}{B(\rho,\frac{n}{c}\rho+1)} \int_t^\infty u^{\rho -1}
	 (1+cu)^{-(\frac{n}{c}+1)\rho -1} du
\\
	& = &
	- \frac{c^\rho}{B(\rho,\frac{n}{c}\rho+1)} 
	\frac{(\rho-1)!}{\Gamma((\frac{n}{c}+1)\rho +1)}  \sum_{i=0}^{\rho-1}  \left .
	\frac{\Gamma(\frac{n}{c}\rho +1+i)}{i!c^{\rho -i}} 
	u^{i}  (1+cu)^{-(\frac{n}{c}\rho +1 +i)} \right |_t^\infty
\\
	& = &
	 \sum_{i=0}^{\rho-1}  
	\frac{(n\rho +c)^{c, \overline{i}}}{i!}
	t^{i}  (1+ct)^{-(\frac{n}{c}\rho +1 +i)} 
\\
	& = &
	\sum_{i=0}^{\rho-1} p_{n\rho+c,i} (t).
\end{eqnarray*}
For $ j \in \mathbb{N} $ we get
\begin{eqnarray*}
	\lefteqn{\omega_{n,j,\rho} (t) }
\\
	& = &
	\frac{c^{j\rho}}{B(j\rho,\frac{n}{c}\rho+1)} \int_0^t u^{j\rho -1}
	 (1+cu)^{-(\frac{n}{c}+j)\rho -1} du
\\
	& &
	- \frac{c^{(j+1)\rho}}{B((j+1)\rho,\frac{n}{c}\rho+1)} \int_0^t u^{(j+1)\rho -1}
	 (1+cu)^{-(\frac{n}{c}+j+1)\rho -1} du
\\
	& = &
	- \frac{c^{j\rho}}{B(j\rho,\frac{n}{c}\rho+1)} 
	\frac{(j\rho -1)!}{\Gamma((\frac{n}{c}+j)\rho +1)}
	\sum_{i=0}^{j\rho-1} \frac{\Gamma(\frac{n}{c}\rho +1+i)}{i!c^{j\rho -i}} \left .
	u^i  (1+cu)^{-(\frac{n}{c}\rho +1+i} \right |_0^t
\\
	& &
	+\frac{c^{(j+1)\rho}}{B((j+1)\rho,\frac{n}{c}\rho+1)}
	\frac{((j+1)\rho -1)!}{\Gamma((\frac{n}{c}+j+1)\rho +1)}
\\
	& & \times
	\sum_{i=0}^{(j+1)\rho-1} \frac{\Gamma(\frac{n}{c}\rho +1+i)}{i!c^{(j+1)\rho -i}} \left .
	u^i  (1+cu)^{-(\frac{n}{c}\rho +1+i} \right |_0^t
\\
	& = &
	 -\sum_{i=0}^{j\rho-1}  
	\frac{(n\rho +c)^{c, \overline{i}}}{i!}
	t^{i}  (1+ct)^{-(\frac{n}{c}\rho +1 +i)} 
	+ \sum_{i=0}^{(j+1)\rho-1}  
	\frac{(n\rho +c)^{c, \overline{i}}}{i!}
	t^{i}  (1+ct)^{-(\frac{n}{c}\rho +1 +i)} 
\\
	& = &
	\sum_{i=0}^{\rho-1} p_{n\rho+c,i+j\rho} (t).
\end{eqnarray*}
\hfill \qed
\end{proof}
With Lemma \ref{lem-int} and (\ref{eq-r1}) we have now the desired representation of $B_{n,\rho}^{(1)} (f;x)$ in terms of the basis functions, i.~e.,
\begin{eqnarray}
\label{rep1}
	B_{n,\rho}^{(1)} (f;x)
	& = & 
	n
	\sum_{j=0}^{\infty} p_{n+c,j} (x) \int_0^\infty  \left \{ \sum_{i=0}^{\rho-1} p_{n\rho+c,i+j\rho } (t)
	\right \}   f(t) dt .
\end{eqnarray}
So, for the kernel function 
$$
	K_{n,j,\rho} (t) = \sum_{i=0}^{\rho-1} p_{n\rho+c,i+j\rho } (t)
$$
we have $ K_{n,j,1} (t)=p_{n+c,j} (t) $ and we conjecture (see also Section \ref{sec7}) that
$$
	\lim_{\rho \to \infty} K_{n,j,\rho}(t) = \left \{
	\begin{array}{cll}
	1 & , & t \in \left [ \frac{j}{n} , \frac{j+1}{n} \right ),
	\\
	0 & , & t \in [0,\infty) \backslash \left [ \frac{j}{n} , \frac{j+1}{n} 	\right ).
\end{array} \right .
$$

\begin{remark}
Observe that the kernel $K_{n,j,\rho} $  can also be written in terms of classical Baskakov type operators applied to the characteristic functions $\chi_{\left [\frac{j\rho}{n\rho+c},\frac{(j+1)\rho}{n\rho+c} \right )} $, i.~e.,
\begin{eqnarray*}
	B_{n\rho+c,\infty} (\chi_{\left [ \frac{j\rho}{n\rho+c},\frac{(j+1)\rho}{n\rho+c} \right )} ;t) 
	& = &
	\sum_{i=0}^{\infty} p_{n\rho+c,i} (t) \chi_{\left [\frac{j\rho}{n\rho+c},\frac{(j+1)\rho}{n\rho+c} \right )}  \left( \frac{i}{n\rho+c} \right )
\\
	& = &
	\sum_{i=j\rho}^{(j+1)\rho -1} p_{n\rho+c,i} (t) 
\\
	& = &
	\sum_{i=0}^{\rho -1} p_{n\rho+c,i+j\rho } (t) .
\end{eqnarray*}
\end{remark}

\section{Representation for the $k$-th order Kantorovich modification}
\label{sec5}
In this section we generalize the representation of the operators to $k \in \mathbb{N}$.

\begin{theorem}
\label{thm2}
Let $n,k \in \mathbb{N}$, $n-k \geq 1$, $\rho \in \mathbb{N}$ and $f \in W_n^\rho$. Then we have the representation
\begin{eqnarray*}
	{B_{n,\rho}^{(k)} (f;x)}
	& = &
	\frac{n^{c,\overline{k}}}{(n\rho)^{c,\underline{k-1}}} \sum_{j=0}^{\infty} p_{n+kc,j} (x)
\\
	& &
	\times 
	\int_0^\infty
	\sum_{i_1=0}^{\rho-1} \dots \sum_{i_k=0}^{\rho-1}
	p_{n\rho-c(k-2),j\rho +i_1 + \dots +i_k+k-1} (t)   f(t) dt .
\end{eqnarray*}
\end{theorem}
\begin{proof}
We prove the theorem by induction.
\\
For $k=1$ see (\ref{rep1}).
\\
$k \Rightarrow k+1$: From the definition of the operators $B_{n,\rho}^{(k+1)}  $ we get
\begin{eqnarray*}
	B_{n,\rho}^{(k+1)} (f;x)
	& = &
	D^1 \circ B_{n,\rho}^{(k)} \circ I_1 (f;x)
\\
	& = &
	\frac{n^{c,\overline{k}}}{(n\rho)^{c,\underline{k-1}}} \sum_{j=0}^{\infty} p_{n+kc,j}' (x)
\\
	& &
	\times 
	\int_0^\infty
	\sum_{i_1=0}^{\rho-1} \dots \sum_{i_k=0}^{\rho-1}
	p_{n\rho-c(k-2),j\rho +i_1 + \dots +i_k+k-1} (t)  I_1( f;t) dt .
\end{eqnarray*}
By using (\ref{eq1.1}) and an appropriate index transform we derive
\begin{eqnarray}
\label{eq5.1}
	B_{n,\rho}^{(k+1)} (f;x)
	& = &
	\frac{n^{c,\overline{k+1}}}{(n\rho)^{c,\underline{k-1}}}
 \sum_{j=0}^{\infty}  p_{n+(k+1)c,j} (x)
\\
\nonumber
	& &
	\times 
	\int_0^\infty
	\sum_{i_1=0}^{\rho-1} \dots \sum_{i_k=0}^{\rho-1}
	\left [ p_{n\rho-(k-2)c,(j+1)\rho +i_1 + \dots +i_k+k-1} (t) \right .
\\
\nonumber
	& &
	\left . \qquad \qquad \qquad \qquad \quad 
	 - p_{n\rho-(k-2)c,j \rho +i_1 + \dots +i_k+k-1} (t)\right ] I_1 (f;t) dt .
\end{eqnarray}
Now we rewrite the difference of the basis functions in the integral and use again (\ref{eq1.1}), i.e.,
\begin{eqnarray*}
\lefteqn{
	p_{n\rho-(k-2)c,(j+1)\rho +i_1 + \dots +i_k+k-1} (t)
	- p_{n\rho-(k-2)c,j \rho +i_1 + \dots +i_k+k-1} (t)}
\\
	& = &
	\sum_{i_{k+1}=0}^{\rho-1}
	\left [ p_{n\rho-(k-2)c,j\rho +i_1 + \dots +i_{k+1}+k} (t) -
	p_{n\rho-(k-2)c,j \rho +i_1 + \dots +i_{k+1}+k-1} (t) \right ]
\\
	& = &
	- \frac{1}{n\rho -(k-1)c}\sum_{i_{k+1}=0}^{\rho-1}
	p_{n\rho-(k-1)c,j\rho +i_1 + \dots +i_{k+1}+k}' (t) .
\end{eqnarray*}
Together with (\ref{eq5.1}) and integration by parts this leads to
\begin{eqnarray*}
	\lefteqn{B_{n,\rho}^{(k+1)} (f;x)}
\\
	& = &
	\frac{n^{c,\overline{k+1}}}{(n\rho)^{c,\underline{k}}}
	\sum_{j=0}^{\infty} p_{n+(k+1)c,j} (x) 
\\
	& &
	\times 
	\int_0^\infty
	\sum_{i_1=0}^{\rho-1} \dots \sum_{i_k=0}^{\rho-1}
	\sum_{i_{k+1}=0}^{\rho-1}
	- p_{n\rho-(k-1)c,j\rho +i_1 + \dots +i_{k+1}+k}' (t)
	 I_1 (f;t) dt
\\
	& = &
	\frac{n^{c,\overline{k+1}}}{(n\rho)^{c,\underline{k}}}
	\sum_{j=0}^{\infty} p_{n+(k+1)c,j} (x) 
\\
	& &
	\times 
	\int_0^\infty
	\sum_{i_1=0}^{\rho-1} \dots 
	\sum_{i_{k+1}=0}^{\rho-1}
	p_{n\rho-(k-1)c,j\rho +i_1 + \dots +i_{k+1}+k} (t) f (t) dt .
\end{eqnarray*}
\hfill \qed
\end{proof}

\section{Convexity of the linking operators}
\label{sec6}
In \cite[Theorem 6]{Paltanea2014} P\u {a}lt\u {a}nea considered convexity properties of the operators $B_{n,\rho}$ in case $c=0$, i.~e., the case of Sz\'{a}sz-Mirakjan linking operators. In a long and tricky proof he showed that
$\left (B_{n,\rho} f \right ) ^{(r)} \geq 0$ for each function $f \in W_n^\rho \cap C^r [0,\infty)$, $f^{(r)} \geq 0$. 

By using P\u {a}lt\u {a}nea's method this  can be generalized to $ B_{n,\rho}^{(k)}$, $ c=0$, $ k \in \mathbb{N}$ 
(see \cite[Theorem 3]{BaumannHeilmannRasa2016}).
From the representation in Theorem \ref{thm2} we know that all the operators $ B_{n,\rho}^{(k)}$, $\rho \in \mathbb{N}$, are positive and the convexity properties for  $ B_{n,\rho}^{(k)}$, $\rho \in \mathbb{N}$ now
follow as a corollary. For $c>0$ this solves an open problem mentioned in  \cite{BaumannHeilmannRasa2016}.
\begin{corollary}
Let $f \in W_n^\rho $ with $f^{(r)} (x) \geq 0$, $ r \in \mathbb{N}_0$, for all $x \in [0,\infty) $. Then
$$
	D^r \left (B_{n,\rho}^{(k)} (f;x) \right ) \geq 0 
$$
for each $ k \in \mathbb{N}$, $x \in [0,\infty)$.
\end{corollary}
\begin{proof}
\begin{eqnarray*}
	D^r B_{n,\rho}^{(k)}  f
	& = & D^r D^k B_{n,\rho} I_k f = D^{k+r}  B_{n,\rho} I_{k+r} f^{(r)}  =
	B_{n,\rho}^{(k+r)} f^{(r)} \geq 0, 
\end{eqnarray*}
as $f^{(r)} \geq 0$.
\hfill \qed
\end{proof}

\section{Remarks on convergence to B-splines}
\label{sec7}
In \cite{HeilmannRasa2017} the authors considered the topic for the linking Bernstein operators with a different approach. In \cite[Theorem 2.3]{GonskaPaltanea2010-1} Gonska and P\u {a}lt\u {a}nea proved the uniform convergence on $[0,1]$ of the linking Bernstein operators to the classical Bernstein operators for every function $ f \in C[0,1]$. With explicit representations for the images of monomials and density arguments this can be extended to the $k$-th order Kantorovich modifications (see \cite[Introduction]{HeilmannRasa2017}). From the representation of the $k$-th order Kantorovich modifications of the linking Bernstein operator and the classical Bernstein operator (see \cite[Theorem 2, (2)]{HeilmannRasa2017}) one can derive immediately that
$$
	\lim_{\rho \to \infty} \frac{1}{\rho^{k-1}} 
	\sum_{i_1=0}^{\rho-1} \dots \sum_{i_k=0}^{\rho-1}
	\tilde{p}_{n\rho+k-2,j\rho +i_1 + \dots +i_k+k-1} (t)  = N_{n,k,j}(t)
$$
for $t \in [0,1]$ with the Bernstein basis functions $\tilde{p}_{n,j}(t) = {n \choose j } t^j (1-t)^{n-j}$.

As the situation concerning the convergence of the operators in the Baskakov-type case is more complicate (see Theorems \ref{thm-c1}, \ref{thm-c2} and \ref{thm-c3}) we can only conjecture that
$$
	\lim_{\rho \to \infty}  \frac{1}{\rho^{k-1}} 
	\sum_{i_1=0}^{\rho-1} \dots \sum_{i_k=0}^{\rho-1}
	p_{n\rho-c(k-2),j\rho +i_1 + \dots +i_k+k-1} (t)  = N_{n,k,j}(t)
$$
for $t \in [0,\infty)$.

We fortify our conjecture by the following illustrations (see Figure \ref{figure1}) where we chose $k=1$, $c=1$, $n=5$ and $j=1$. 

\begin{figure}[ht]
 \centering
 \includegraphics[scale=0.5]{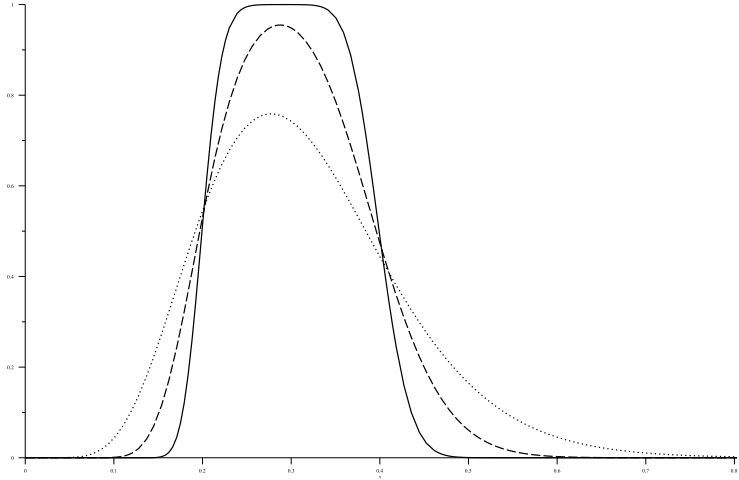}
 \caption{The dotted line belongs to $\rho = 10$, the dashed  line to $\rho = 30$ and the solid line to $\rho = 150$.}
 \label{figure1}
\end{figure}


Margareta Heilmann\\
School of Mathematics and Natural Sciences
\\
University of Wuppertal
\\
Gau{\ss}stra{\ss}e 20
\\
D-42119 Wuppertal, Germany
\\
email: {heilmann@math.uni-wuppertal.de}
\\[5pt]
Ioan Ra\c{s}a
\\
Department of Mathematics
\\
Technical University
\\
Str. Memorandumului 28
\\
RO-400114 Cluj-Napoca,
Romania
\\
email: {Ioan.Rasa@math.utcluj.ro}
\end{document}